\documentclass[12pt,reqno]{amsart}

\newcommand\version{October 25, 2026}

\usepackage{amsmath,amsfonts,amsthm,amssymb,amsxtra,color}
\usepackage{bbm} 
\usepackage{hyperref} 
\usepackage{booktabs}

\newtheorem{theorem}{Theorem}
\newtheorem{proposition}[theorem]{Proposition}
\newtheorem{lemma}[theorem]{Lemma}

\theoremstyle{definition}

\theoremstyle{remark}

\numberwithin{equation}{section}

 \newcommand{\dd}{\,\mathrm{d}}
 
\renewcommand{\epsilon}{\varepsilon}
\newcommand{\g}{\gamma_*}

\renewcommand{\phi}{\varphi}
\newcommand{\R}{\mathbb{R}}

\newcommand{\Ima}{\operatorname{Im}}

\newcommand{\pos}[1]{\left(#1\right)_+}

\DeclareMathOperator{\im}{Im}

\DeclareMathOperator{\re}{Re}

\DeclareMathOperator{\Tr}{Tr}

\begin{document}

\title[Optimal Riesz mean inequalities for the Coulomb Hamiltonian]{Optimal Riesz mean inequalities\\ for the Coulomb Hamiltonian}

\begin{abstract}
	We are interested in whether the Riesz means of the eigenvalues of the Coulomb Hamiltonian can be bounded by the corresponding semiclassical phase space integral. In dimensions two and three we determine the optimal Riesz exponent above which this bound holds.
\end{abstract}

\author{Rupert L. Frank}
\address[Rupert L. Frank]{Mathe\-matisches Institut, Ludwig-Maximilians Universit\"at M\"unchen, The\-resienstr.~39, 80333 M\"unchen, Germany, and Munich Center for Quantum Science and Technology, Schel\-ling\-str.~4, 80799 M\"unchen, Germany}
\email{r.frank@lmu.de}

\author{Bernard Helffer}
\address[Bernard Helffer]{Laboratoire de Math\'ematiques Jean Leray, CNRS, Nantes Universit\'e, F44000. Nantes, France}
\email{Bernard.Helffer@univ-nantes.fr}

\author{Fran\c{c}ois Nicoleau}
\address[Fran\c{c}ois Nicoleau]{Laboratoire de Math\'ematiques Jean Leray, CNRS, Nantes Universit\'e, F44000. Nantes, France}
\email{francois.nicoleau@univ-nantes.fr}

\date{\version}
\thanks{\copyright\, 2026 by the authors. This paper may be reproduced, in its entirety, for non-commercial purposes.\\
	Partial support through the German Research Foundation through EXC-2111-390814868 and TRR 352-Project-ID 470903074 (R.L.F.) is acknowledged.}

\maketitle
		
	
	\section{Introduction and main results}\label{sec:main}
	
	We study Riesz means of the eigenvalues of the Coulomb Hamiltonian
	\[
	-h^2\Delta - \kappa\, |x|^{-1}
	\qquad\text{in}\ L^2(\R^d) \,.
	\]
	Here $h$ and $\kappa$ are positive constants and $d\geq 2$ is the spatial dimension. More precisely, we ask for which exponents $\gamma>0$ the Riesz means
	\[
	\Tr(-h^2\Delta- \kappa\,|x|^{-1} +\mu)_-^\gamma 
	\]
	can be bounded from above by their semiclassical analogue
	\[
	\iint_{\R^d\times\R^d} \left( h^2 |\xi|^2 - \kappa\,|x|^{-1} + \mu \right)_-^\gamma \,\frac{\dd x \dd \xi}{(2\pi)^d} \,.
	\]
	We want to have this bound uniformly for all $h,\kappa$ and $\mu>0$. Here and throughout, we use the notation $t_\pm:=\max\{0,\pm t\}$.
	
	For fixed $h,\kappa,\mu>0$, the spectral trace is finite for every $\gamma>0$: there are only finitely many eigenvalues of $-h^2\Delta - \kappa |x|^{-1}$ below $-\mu$. The phase-space integral, however, is finite precisely for $\gamma<d/2$, owing to the Coulomb singularity at $x=0$. We therefore restrict ourselves to $0<\gamma<d/2$ in what follows. Lieb--Thirring inequalities bound the spectral trace by \emph{a constant times} the semiclassical analogue; our question is whether that constant can be \emph{chosen to be one}.
	
	This question has been settled in all dimensions $d\geq 4$. Indeed, Carvalho Corso, Weidl and Zeng \cite{CWZ} have shown that such a bound with constant one is valid for $\gamma\geq1$, while our companion work \cite{FHN} establishes failure for $0<\gamma<1$. The two- and three-dimensional cases remain open in those works; we settle them here. Their thresholds turn out to differ from the higher-dimensional threshold.
	
	The following two theorems are our main results.
	
	\begin{theorem}\label{main3d}
		Let $d=3$.
		\begin{enumerate}
			\item[(a)] If $2/3\leq\gamma<3/2$, then, for all $h,\kappa,\mu>0$, we have
			\[
			\Tr(-h^2\Delta- \kappa\,|x|^{-1} + \mu)_-^\gamma 
			< \iint_{\R^3\times\R^3} \left( h^2 |\xi|^2 - \kappa\,|x|^{-1} + \mu \right)_-^\gamma \,\frac{\dd x \dd \xi}{(2\pi)^3} \,.
			\]
			\item[(b)] If $0<\gamma<2/3$ there is a sequence $(\epsilon_n)$ of positive numbers tending to zero such that for any $n$ and all $h,\kappa,\mu>0$ with $h^2 \mu/\kappa^2= \epsilon_n$, we have
			\[
			\Tr(-h^2\Delta- \kappa\,|x|^{-1} + \mu)_-^\gamma 
			> \iint_{\R^3\times\R^3} \left( h^2 |\xi|^2 - \kappa\,|x|^{-1} + \mu \right)_-^\gamma \,\frac{\dd x \dd \xi}{(2\pi)^3} \,.
			\]
		\end{enumerate}
	\end{theorem}
	
\begin{theorem}\label{main2d}
	Let $d=2$. There is a unique $\gamma_*\in(0,1)$ such that:
	\begin{enumerate}
		\item[(a)] If $\gamma_*<\gamma<1\,$, then, for all $h,\kappa,\mu>0$, we have
		\[
		\Tr(-h^2\Delta- \kappa\,|x|^{-1} +\mu)_-^\gamma
		< \iint_{\R^2\times\R^2} \left( h^2 |\xi|^2 - \kappa\,|x|^{-1} + \mu \right)_-^\gamma \,\frac{\dd x \dd \xi}{(2\pi)^2} \,.
		\]
		\item[(b)] If $0<\gamma\leq\gamma_*$, then, for all $h,\kappa,\mu>0$, we have
		\[
		\Tr(-h^2\Delta- \kappa\,|x|^{-1} +\mu)_-^\gamma
		\leq \frac{4\sin(\pi\gamma)}{\pi\gamma}\ \gamma^{\gamma} (1-\gamma)^{1-\gamma} \iint_{\R^2\times\R^2} \left( h^2 |\xi|^2 - \kappa\,|x|^{-1} + \mu \right)_-^\gamma \,\frac{\dd x \dd \xi}{(2\pi)^2} \,.
		\]
		with equality if and only if $h^2\mu/\kappa^2 = 1-\gamma$. In particular, if $0<\gamma<\gamma_*$, then for all $h,\kappa,\mu>0$ with $h^2\mu/\kappa^2 = 1-\gamma$ we have
		\[
		\Tr(-h^2\Delta- \kappa\,|x|^{-1} + \mu)_-^\gamma
		> \iint_{\R^2\times\R^2} \left( h^2 |\xi|^2 - \kappa\,|x|^{-1} + \mu \right)_-^\gamma \,\frac{\dd x \dd \xi}{(2\pi)^2} \,.
		\]
	\end{enumerate}
	Moreover, $\gamma=\gamma_*$ is the unique solution in the interval $(0,1)$ of
	\begin{equation}\label{eq:root}
		\gamma^{\gamma} (1-\gamma)^{1-\gamma}
		=\frac{\pi\gamma}{4\sin(\pi\gamma)} \,.
	\end{equation}
\end{theorem}
	
	Numerically, we find that $\gamma_* \approx 0.61265$\,.

	Let us give some background on semiclassical eigenvalue bounds. Such bounds already played an important role in the work of Birman, Solomyak and Rozenblum (see, for example, \cite{BS,Ro,Ro2}) and received a new impetus by the Lieb--Thirring proof of stability of matter \cite{LT1}. The latter work has put an emphasis on the value of the constants in these bounds and has led to the celebrated Lieb--Thirring conjecture \cite{LT2}. For the background on this problem we refer to the textbook \cite{FLW} and the survey \cite{Fr}. Recent breakthroughs are \cite{RS} (where, among other things, the currently best known constants for $\gamma=1$ in the higher dimensional case are obtained) and the resolution of the one-dimensional Lieb--Thirring conjecture in \cite{OAI}. Despite this progress, many important cases of the Lieb--Thirring conjecture in the higher dimensional case are still open.
	
	In the absence of a proof of the Lieb--Thirring conjecture, there has been some activity directed towards investigating the analogue of this conjecture for certain subclasses of potentials. This includes, in particular, the harmonic oscillator, studied in \cite{HeRo3,dlB}, where it was shown that, in any dimension $d$, the phase space integral is an upper bound on the Riesz means (with constant one) if and only if $\gamma\geq 1$; see also \cite[Section 5.2]{FLW}. The same threshold phenomenon at $\gamma=1$ holds for the Coulomb Hamiltonian in dimensions $d\geq 4$, as recently shown in \cite{CWZ} (case $\gamma\geq 1$) and \cite{FHN} (case $\gamma<1$). In view of these results we find it quite remarkable that, as shown by our Theorems \ref{main3d} and \ref{main2d}, in dimensions two and three, the threshold does \emph{not} occur at $\gamma=1$. Other related results where $\gamma=1$ appears as a threshold concern the case of a constant magnetic field; see \cite{ELV} (case $\gamma\geq 1$) and \cite{FLoW} (case $\gamma<1$). Meanwhile, a recent result where the threshold $\gamma=1$ was broken is \cite{FL}.
	
	What makes the value $\gamma=1$ special is the convexity of $E\mapsto E_-^\gamma$, which holds if and only if $\gamma\geq 1$. This convexity enters explicitly or implicitly in several of the above mentioned proofs for $\gamma\geq 1$. Conversely, it is this absence of convexity that makes our proofs rather involved. 
	
	Making use of the explicit form of the eigenvalues of the Coulomb Hamiltonian, we can rewrite the question of whether the Riesz means are bounded by the semiclassical phase space integral as a question of whether a certain Riemann sum approximation is less than or equal to the corresponding integral. The nodes of the Riemann sum are equidistant and the bound is required to be uniform in the mesh size. One of the main analytic ingredients is an Abel--Plana-type integral representation formula for the difference between the Riemann sum and the corresponding integral; see Lemmas \ref{lem:representation3} and \ref{lem:representation2}. This reduces the problem to estimating certain integrals. Most of the estimates are analytic, and only a few elementary numerical bounds are needed at the end. These can, in principle, be verified using rational arithmetic.
	
Finally, let us comment on the difference between the two- and the three-dimensional case. In Subsection \ref{sec:gammastar} we will see that $\gamma\mapsto 
\frac{4\sin(\pi\gamma)}{\pi\gamma}\ \gamma^{\gamma} (1-\gamma)^{1-\gamma}$ is strictly monotone, so the two-dimensional inequality can be written as
\begin{equation}\label{eq:main2d}
	\begin{aligned}
		\Tr(-h^2\Delta- \kappa\,|x|^{-1} +\mu)_-^\gamma 
		& \leq \max\left\{ 1, \frac{4\sin(\pi\gamma)}{\pi\gamma}\ (1-\gamma)^{1-\gamma}\gamma^{\gamma} \right\} \\
		& \quad\times \iint_{\R^2\times\R^2} \left( h^2 |\xi|^2 - \kappa\,|x|^{-1} + \mu \right)_-^\gamma \,\frac{\dd x \dd \xi}{(2\pi)^2} \,.
	\end{aligned}
\end{equation}
An elementary computation shows that $C=\frac{4\sin(\pi\gamma)}{\pi\gamma}\ (1-\gamma)^{1-\gamma}\gamma^{\gamma}$ is the best constant in the bound
\[
\left( \frac{\kappa^2}{(d-1)^{2}h^2} - \mu \right)_-^\gamma \leq C \iint_{\R^2\times\R^2} \left( h^2 |\xi|^2 - \kappa\,|x|^{-1} + \mu \right)_-^\gamma \,\frac{\dd x \dd \xi}{(2\pi)^2}
\]
for the \emph{smallest eigenvalue} $- h^{-2} \kappa^2 (d-1)^{-2}$ of $-h^2\Delta -\kappa |x|^{-1}$. Thus, the meaning of \eqref{eq:main2d} is that the optimal constant for the full trace is either given by the constant that is required to control a single eigenvalue or by the semiclassical constant. This is in the spirit of the original conjecture by Lieb and Thirring \cite{LT2}. While we do not prove this here, the three-dimensional case is rather different. Numerical experiments suggest that the optimal constant is determined by the lowest energy level up to \(\gamma_1\approx0.6028\), by the two lowest energy levels up to \(\gamma_2\approx0.6134\), and so on. The number of energy levels determining the optimal constant appears to diverge as \(\gamma\) approaches \(2/3\) from below. We do not pursue this question further in this paper.

	
	\section{Initial reductions}\label{sec:reductions}
	
	The main part of this paper is concerned with proving Theorems~\ref{main3d} and~\ref{main2d} for $\gamma = 2/3$ and $\gamma=\gamma_*$, respectively. In this preliminary section, we explain that all other assertions can be easily deduced from those. Specifically, we provide
	\begin{itemize}
		\item the reduction of part (a) of Theorems~\ref{main3d} and~\ref{main2d} to the respective critical exponents $2/3$ and $\gamma_*$ (see Subsection \ref{sec:above}),
		\item the proof of part (b) in Theorem \ref{main3d} (see Subsection \ref{sec:below3}), and
		\item the proof of part (b) in Theorem \ref{main2d} (see Subsection \ref{sec:below2}).
	\end{itemize} 
	
	
	\subsection{Proof of part (a) above the critical exponent}\label{sec:above}
	We treat the two and three-dimensional case together and write $\gamma_*= 2/3$ when $d= 3$. We assume that the assertion of Theorems \ref{main3d} and \ref{main2d} hold for $\gamma=\gamma_*$ and show that this implies the assertion in those theorems for $\gamma_*<\gamma<d/2$.
	
	We apply the Aizenman--Lieb argument \cite{AL}, which is based on the identity
	\[
	E_-^\gamma=c_{\gamma,\sigma}\int_0^\infty
	(E+s)_-^\sigma s^{\gamma-\sigma-1}\dd s \,,
	\]
	valid for all $0<\sigma<\gamma$ and $E\in\R$ with a positive constant $c_{\gamma,\sigma}\,$. We assume that the inequality holds at the critical exponent $\gamma_*\,$. For $\gamma_*<\gamma<d/2$ we apply the above identity twice with $\sigma=\gamma_*$ and use the assumed inequality for $\gamma_*$ to obtain
	\begin{align*}
		& \Tr(-h^2\Delta- \kappa\,|x|^{-1} + \mu)_-^\gamma \\
		& = c_{\gamma,\gamma_*} \int_0^\infty \Tr(-h^2\Delta- \kappa\,|x|^{-1} + \mu+\mu')_-^{\gamma_*} (\mu')^{\gamma-\gamma_*-1}\,d\mu' \\
		& < c_{\gamma,\gamma_*} \int_0^\infty \iint_{\R^d\times\R^d} \left( h^2 |\xi|^2 - \kappa\,|x|^{-1} + \mu+\mu' \right)_-^{\gamma_*} \,\frac{\dd x \dd \xi}{(2\pi)^d} (\mu')^{\gamma-\gamma_*-1}\,d\mu' \\
		& = \iint_{\R^d\times\R^d} c_{\gamma,\gamma_*} \int_0^\infty \left( h^2 |\xi|^2 - \kappa\,|x|^{-1} + \mu+ \mu' \right)_-^{\gamma_*} \,(\mu')^{\gamma-\gamma_*-1}\,d\mu' \, \frac{\dd x \dd \xi}{(2\pi)^d} \\
		& = \iint_{\R^d\times\R^d} \left( h^2 |\xi|^2 - \kappa\,|x|^{-1} + \mu \right)_-^{\gamma} \, \frac{\dd x \dd \xi}{(2\pi)^d} \,.
	\end{align*}
	The inequality is strict, since in dimension three the critical inequality is strict for every shifted parameter \(\mu+\mu'\) and in dimension two, equality at the critical exponent is possible for at most one value of \(\mu'\), which has measure zero. This proves the claimed inequality for~$\gamma$.
	\qed
	
	
	\subsection{Scaling}\label{sec:scaling}
	
	Before turning to the proof of parts (b) in Theorems \ref{main3d} and \ref{main2d}, we use the fact that by scaling and homogeneity of Riesz means, two of the three parameters $h,\kappa,\mu$ can be fixed. For us, it will be convenient to fix $\mu=1$ and $\kappa=2$. Indeed, changing variables $x=\kappa y/(2\mu)$, we find
	\[
	\Tr(-h^2\Delta_x - \kappa |x|^{-1}+\mu)_-^\gamma = \mu^\gamma \Tr(- \tfrac{4\mu h^2}{\kappa^2} \Delta_y - 2|y|^{-1} + 1)_-^\gamma
	\]
	and, similarly,
	\[
	\iint_{\R^d\times\R^d} \left( h^2 |\xi|^2 - \kappa\,|x|^{-1} + \mu \right)_-^{\gamma} \, \frac{\dd x \dd \xi}{(2\pi)^d} = \mu^\gamma \iint_{\R^d\times\R^d} \left( \tfrac{4 \mu h^2}{\kappa^2} |\eta|^2 - 2\,|y|^{-1} + 1 \right)_-^{\gamma} \, \frac{\dd y \dd \eta}{(2\pi)^d} \,.
	\]
	Thus we may choose $\mu=1$ and $\kappa=2$, with the new parameter $h_{\mathrm{eff}}=2h\sqrt{\mu}/\kappa$. Below we again write $h$ for $h_{\mathrm{eff}}$ and $x$ for the rescaled spatial variable.
	
	
	\subsection{Proof of part (b) in Theorem \ref{main3d}}\label{sec:below3}
	
	The following argument is similar in spirit to an argument by Helffer and Robert \cite{HeRo3}. It is based on the asymptotic expansion from \cite{FHN}.
	
	Let $d=3$ and $0<\gamma<2/3$. The asymptotics from \cite{FHN} yield
	\begin{align*}
		\Tr(-h^2\Delta - 2|x|^{-1} +1)_-^\gamma & = \iint_{\R^3\times\R^3}  \left( h^2 |\xi|^2 - 2\,|x|^{-1} + 1 \right)_-^{\gamma} \, \frac{\dd x \dd \xi}{(2\pi)^3} \\
		& \quad + \sigma_\gamma(h^{-1}) h^{-2+\gamma} + o(h^{-2+\gamma}) \,.
	\end{align*}
	Here $\sigma_\gamma$ is an (explicit) continuous, nonzero, 1-periodic function of mean zero. It follows that
	\begin{align*}
		& \limsup_{h\to 0} h^{2-\gamma} \left( \Tr(-h^2\Delta - 2|x|^{-1} +1)_-^\gamma - \iint_{\R^3\times\R^3}  \left( h^2 |\xi|^2 - 2\,|x|^{-1} + 1 \right)_-^{\gamma} \, \frac{\dd x \dd \xi}{(2\pi)^3} \right) \\
		& = \max \sigma_\gamma > 0 \,.
	\end{align*}
	This proves that the semiclassical inequality does not hold and, moreover, that it is violated along a sequence $h_n = 1/(\phi_\gamma+n)$, where $\phi_\gamma\in[0,1]$ is such that $\sigma_\gamma$ attains its maximum at $\phi_\gamma$ and $n\in\mathbb N$ is sufficiently large. Returning to the original parameters, choose $\epsilon_n=h_n^2/4$ (and reindex to omit finitely many initial values). This gives the assertion in part (b).
	\qed
	
	
	\subsection{The definition of \texorpdfstring{$\gamma_*$}{gammastar} in Theorem \ref{main2d}}\label{sec:gammastar}
	
	In order to show that equation \eqref{eq:root} has a unique solution $\gamma=\gamma_*\in(0,1)$, we will use the following simple lemma, which we will also use again later.
	
	\begin{lemma}\label{lem:powerelem}
		For any $a\in\R$ and any $0<\gamma<\eta$,
		\begin{equation}\label{eq:power}
			\pos{1+\frac a\gamma}^{\gamma}
			\le \pos{1+\frac a\eta}^{\eta}.
		\end{equation}
		If $1+a/\gamma>0$, equality holds if and only if $a=0$.
	\end{lemma}
	\begin{proof}
		If $1+a/\gamma\le0$, the left side vanishes. Otherwise, putting
		$\lambda:=\gamma/\eta\in(0,1)$ and $z:=1+a/\gamma>0$, the arithmetic-geometric mean inequality gives
		\[
		1+\frac a\eta=(1-\lambda)+\lambda z\ge z^\lambda.
		\]
		Raising to the power $\eta$ proves \eqref{eq:power}. Equality in the arithmetic-geometric mean inequality requires $z=1$, equivalently $a=0$.
	\end{proof}
	
	For $0<\gamma<1$ and $r>0$, we set
	\[
	K_\gamma(r):=r\pos{1+\frac{r^{-1}-1}{\gamma}}^{\gamma}.
	\]
	The substitution $z=(1-\gamma)r$ yields
	\begin{align}\label{eq:integralk}
		\int_0^\infty\frac{K_\gamma(r)}r\dd r
		&=\frac{1}{\gamma^\gamma}
		\int_0^{1/(1-\gamma)}r^{-\gamma}
		(1-(1-\gamma)r)^\gamma\dd r
		= \frac1{\gamma^\gamma(1-\gamma)^{1-\gamma}} \int_0^1 z^{-\gamma}(1-z)^\gamma\dd z \nonumber\\
		& =\frac{\frac{\pi\gamma}{\sin(\pi\gamma)}}{\gamma^\gamma(1-\gamma)^{1-\gamma}} \,.
	\end{align}
	In the last equality, we first expressed the integral as a beta function and then simplified it using the reflection formula for the gamma function. 
	
	By Lemma~\ref{lem:powerelem}, $\gamma\mapsto K_\gamma(r)$ is nondecreasing for every $r>0$, and strictly increasing for every $0<r<1$. It follows that the left side of \eqref{eq:integralk}, and therefore also the right side, are strictly increasing in $\gamma$.
	
	Moreover, the function $\gamma\mapsto \frac{\frac{\pi\gamma}{\sin(\pi\gamma)}}{\gamma^\gamma(1-\gamma)^{1-\gamma}}$ is continuous on $(0,1)$ and satisfies
	\[
	\lim_{\gamma\downarrow0}\frac{\frac{\pi\gamma}{\sin(\pi\gamma)}}{\gamma^\gamma(1-\gamma)^{1-\gamma}}=1 \,,
	\qquad
	\lim_{\gamma\uparrow1}\frac{\frac{\pi\gamma}{\sin(\pi\gamma)}}{\gamma^\gamma(1-\gamma)^{1-\gamma}}=+\infty.
	\]
	Hence the equation $\gamma^\gamma(1-\gamma)^{1-\gamma}=\frac{\pi\gamma}{4\sin(\pi\gamma)}$ has exactly one solution in $(0,1)$,
	and $\gamma^\gamma(1-\gamma)^{1-\gamma}>\frac{\pi\gamma}{4\sin(\pi\gamma)}$ precisely for $\gamma<\gamma_*$.

	\subsection{Relation to Riemann sum approximations to integrals}\label{sec:rewriting}
	
	We now rewrite the Riesz means and the integral in such a way that the inequalities of interest can be understood as a bound on a Riemann sum by the corresponding integral.
	
	We begin with the Riesz means. The discrete eigenvalues of the operator $-h^2\Delta -2|x|^{-1}$ in $L^2(\R^d)$ are $- 4h^{-2}(2n+d-1)^{-2}$, indexed by $n\in\mathbb N=\{0,1,2,\ldots\}$, of multiplicity 
	$$
	\frac{(n+d-2)!\,(2n+d-1)}{(d-1)!\, n!} \,;
	$$
	see, for example, \cite[Proposition 4.10]{FLW}. Therefore the Riesz means are
	$$
	\Tr(-h^2\Delta - 2 |x|^{-1} + 1)_-^\gamma = \sum_{n=0}^\infty \frac{(n+d-2)!\,(2n+d-1)}{(d-1)!\, n!} \left( h^{-2} 4 (2n+d-1)^{-2} - 1 \right)_+^\gamma \,.
	$$
	For $d=2$, this becomes
	\begin{equation}
		\label{eq:trace2d}
		\Tr(-h^2\Delta - 2 |x|^{-1} + 1)_-^\gamma = 2h^{-1}  \sum_{n=0}^\infty (h(n+\tfrac12))^{1-2\gamma} \left( 1 - h^{2} (n+\tfrac12)^{2} \right)_+^\gamma\,,
	\end{equation}
	and for $d=3$,
	\begin{equation}
		\label{eq:trace3d}
		\Tr(-h^2\Delta - 2 |x|^{-1} + 1)_-^\gamma = h^{-2} \sum_{n=0}^\infty (h(n+1))^{2-2\gamma} \left( 1- h^{2} (n+1)^{2} \right)_+^\gamma \,.
	\end{equation}
	
	Next, we rewrite the phase-space integral as
	\begin{equation}
		\label{eq:phasespacealt}
		\iint_{\mathbb R^d\times\mathbb R^d}\Big(h^2|\xi|^2-\tfrac{2}{|x|}+1\Big)_-^{\gamma}\frac{\dd x\dd\xi}{(2\pi)^d}
		=\frac{2}{(d-1)!\,h^d}\int_0^1 t^{d-1-2\gamma}(1-t^2)^\gamma\dd t \,.
	\end{equation}
	This can be justified by writing both sides in terms of gamma functions. For the sake of completeness, let us give a short somewhat more conceptual proof, which shows that
	\[
	\iint_{\mathbb R^d\times\mathbb R^d} f\Big(h^2|\xi|^2-\tfrac{2}{|x|}\Big) \, \frac{\dd x\dd\xi}{(2\pi)^d}
	=\frac{2}{(d-1)!\,h^d}\int_0^\infty t^{d-1} f(-t^{-2}) \dd t\,,
	\]
	for any nonnegative measurable function $f$ supported on $(-\infty,0]$. Indeed, we change variables 
	\[
	(x,\xi)\mapsto (t,\omega,\eta) := \left( ( 2|x|^{-1} - h^2|\xi|^2)^{-1/2} \,,\ x/|x| \,,\ ( 2|x|^{-1} - h^2|\xi|^2)^{-1/2} h \xi \right)\,,
	\]
	with inverse
	\[
	(t,\omega,\eta) \mapsto (x,\xi) := \left( \frac{2t^2}{1+|\eta|^2} \, \omega \,,\ \frac{\eta}{ht} \right).
	\]
	This is a bijection between $\{ (x,\xi):\ 2|x|^{-1}>h^2|\xi|^2\}$ and $(0,\infty)\times\mathbb S^{d-1}\times\R^d$. We compute
	\[
	\dd x \dd\xi =h^{-d} 2^{d+1}\,t^{d-1}\dd t \dd \omega \frac{\dd\eta}{(1+|\eta|^2)^{d+1}}
	\]
	and deduce that
	\[
	\iint_{\mathbb R^d\times\mathbb R^d} f\Big(h^2|\xi|^2-\tfrac{2}{|x|}\Big) \, \frac{\dd x\dd\xi}{(2\pi)^d} = \frac{2^{d+1} |\mathbb S^{d-1}|}{(2\pi h)^{d}} \int_0^\infty f(-t^{-2}) \,t^{d-1}\dd t \int_{\R^d} \frac{\dd\eta}{(1+|\eta|^2)^{d+1}} \,.
	\]
	Finally, we use
	\[
	\frac{2^{d} |\mathbb S^{d-1}|}{(2\pi)^{d}} \int_{\R^d} \frac{\dd\eta}{(1+|\eta|^2)^{d+1}} = \frac{1}{(d-1)!}\,,
	\]
	to arrive at the claimed formula.
	
	
	\subsection{Proof of part (b) in Theorem \ref{main2d}}\label{sec:below2}
	
	Let $d=2$ and set
	\begin{align*}
		S_\gamma(h)&:=h\sum_{n=0}^{\infty}
		\bigl(h(n+\tfrac12)\bigr)^{1-2\gamma}
		\pos{1-h^2(n+\tfrac12)^2}^{\gamma} \,.
	\end{align*}
	Our goal is to prove that if the inequality
	\begin{equation}
		\label{eq:goal2dsmall}
		S_\gamma(h) \leq 2 \, \gamma^\gamma(1-\gamma)^{1-\gamma}
	\end{equation}
	holds for $\gamma=\gamma_*$ with equality precisely for $h=2\sqrt{1-\gamma_*}$, then the corresponding inequality holds for all $0<\gamma<\gamma_*$ with equality precisely for $h=2\sqrt{1-\gamma}$. In view of \eqref{eq:trace2d}, \eqref{eq:phasespacealt} and a computation of the corresponding integral, inequality \eqref{eq:goal2dsmall} establishes the inequality in part (b) of Theorem \ref{main2d} for $0<\gamma<\gamma_*$ under the assumption that the inequality holds for $\gamma=\gamma_*$.
	
	The key in proving \eqref{eq:goal2dsmall} is the observation that for every $0<\gamma<\eta<1$ and $h>0$, we have
	\begin{equation}\label{eq:comparison}
		\frac{S_\gamma(h)}{\gamma^\gamma(1-\gamma)^{1-\gamma}}
		\le\frac{S_\eta\!\left(h\sqrt{\frac{1-\eta}{1-\gamma}}\right)}{\eta^\eta(1-\eta)^{1-\eta}} \,.
	\end{equation}
	To see this, we change variables $k=2n+1$ to find
	\[
	S_\gamma(h)
	=2(h^2/4)^{1-\gamma}\sum_{\substack{k\ge1\\k\ \mathrm{odd}}}
	k\pos{k^{-2}-h^2/4}^{\gamma} \,.
	\]
	Putting $h^2/4=(1-\gamma)x$ and dividing by $\gamma^\gamma(1-\gamma)^{1-\gamma}$, we obtain
	\begin{equation}\label{eq:scaled}
		\frac{S_\gamma(2\sqrt{(1-\gamma)x})}{\gamma^\gamma(1-\gamma)^{1-\gamma}}
		=2x\sum_{\substack{k\ge1\\k\ \mathrm{odd}}}
		k\pos{1+\frac{k^{-2}x^{-1}-1}{\gamma}}^{\gamma}.
	\end{equation}
	Applying Lemma~\ref{lem:powerelem} to each term in the sum, we obtain \eqref{eq:comparison}.
	
	To prove \eqref{eq:goal2dsmall} for $0<\gamma<\gamma_*$ we apply \eqref{eq:comparison} with $\eta=\gamma_*$ and find
	\begin{equation}\label{eq:below2proof}
	\frac{S_\gamma(h)}{\gamma^\gamma(1-\gamma)^{1-\gamma}}
	\le\frac{S_{\gamma_*}\!\left(h\sqrt{\frac{1-\gamma_*}{1-\gamma}}\right)}{\gamma_*^{\gamma_*}(1-\gamma_*)^{1-{\gamma_*}}} \,.
	\end{equation}
	By the assumed inequality for $\gamma_*$, the right side is $\leq 2$. This proves \eqref{eq:goal2dsmall}.
	
	Finally, let us discuss the cases of equality in \eqref{eq:goal2dsmall} for $0<\gamma<\gamma_*$. Clearly, equality holds for $h=2\sqrt{1-\gamma}$ (where only one term contributes to the sum). Conversely, if equality holds in \eqref{eq:goal2dsmall} for some $h>0$ and some $0<\gamma<\gamma_*$, then inequality \eqref{eq:below2proof} implies that $S_{\gamma_*}\!\left(h\sqrt{\frac{1-\gamma_*}{1-\gamma}}\right) = 2 \,\gamma_*^{\gamma_*}(1-\gamma_*)^{1-{\gamma_*}}$. Since we are working under the assumption that equality in \eqref{eq:goal2dsmall} for $\gamma=\gamma_*$ holds only at $2\sqrt{1-\gamma_*}$, we deduce that
	\[
	h\sqrt{\frac{1-\gamma_*}{1-\gamma}} = 2\sqrt{1-\gamma_*} \,,
	\]
	that is, $h = 2\sqrt{1-\gamma}$. This proves the assertion.
	\qed

	\subsection{Reduction to two theorems about Riemann sums}\label{sec:reduction}
	
	The discussion so far in this section has reduced the proof of Theorems \ref{main3d} and \ref{main2d} to the proof of the semiclassical inequality for $\gamma=2/3$ if $d=3$ and $\gamma=\gamma_*$ if $d=2$. Indeed, in Subsection \ref{sec:above} we have reduced the case of $\gamma$ larger than the critical exponent to the case of the critical exponent, and in Subsections \ref{sec:below3} (for $d=3$) and \ref{sec:below2} (for $d=2$) we have reduced the case of $\gamma$ smaller than the critical exponent to the case of the critical exponent.
	
	Thus, what remains to be proved is the inequality at the critical exponent. More precisely, what we want to prove are the inequalities for $d=3$
	\begin{equation}
		\label{eq:goal3d}
		\Tr(-h^2\Delta- 2\,|x|^{-1} + 1)_-^{2/3} 
		< \iint_{\R^3\times\R^3} \left( h^2 |\xi|^2 - 2\,|x|^{-1} + 1 \right)_-^{2/3} \,\frac{\dd x \dd \xi}{(2\pi)^3} \,,
	\end{equation}
	and for $d=2$
	\begin{equation}
		\label{eq:goal2d}
		\Tr(-h^2\Delta- 2\,|x|^{-1} + 1)_-^{\gamma_*} 
		\leq \iint_{\R^2\times\R^2} \left( h^2 |\xi|^2 - 2\,|x|^{-1} + 1 \right)_-^{\gamma_*} \,\frac{\dd x \dd \xi}{(2\pi)^2} \,,
	\end{equation}
	with equality in \eqref{eq:goal2d} if and only if $h = 2\sqrt{1-\gamma_*}$.
	
	In view of \eqref{eq:trace2d}, \eqref{eq:trace3d} and \eqref{eq:phasespacealt}, these two inequalities are equivalent to the inequalities
	\begin{equation}
		\label{eq:goal3d2}
		h \sum_{n=0}^\infty (h(n+1))^{2/3} \left( 1- h^{2} (n+1)^{2} \right)_+^{2/3}
		< \int_0^1 t^{2/3}(1-t^2)^{2/3} \dd t
	\end{equation}
	and
	\begin{equation}
		\label{eq:goal2d2}
		h  \sum_{n=0}^\infty (h(n+\tfrac12))^{1-2\gamma_*} \left( 1 - h^{2} (n+\tfrac12)^{2} \right)_+^{\gamma_*}
		\leq \int_0^1 t^{1-2\gamma_*}(1-t^2)^{\gamma_*} \dd t \,,
	\end{equation}
	respectively, with equality in \eqref{eq:goal2d2} if and only if $h = 2\sqrt{1-\gamma_*}$.
	
	The following two theorems, which are the technical main results of this work, imply that the latter two inequalities are true. Consequently, their proof also finishes the proof of Theorems \ref{main3d} and \ref{main2d}.
	
	\begin{theorem}\label{thm:endpoint}
		For every $h>0\,$, we have the strict inequality
		\begin{equation}\label{eq:endpoint-main}
			h\sum_{n=0}^{\infty}(h(n+1))^{2/3}
			\bigl(1-h^2(n+1)^2\bigr)_+^{2/3}
			< \int_0^1 t^{2/3}(1-t^2)^{2/3}\dd t \,.
		\end{equation}
	\end{theorem}
	
	\begin{theorem}\label{thm:threshold}
		For every $h>0$, we have
		\begin{equation}\label{eq:iff}
			h  \sum_{n=0}^\infty (h(n+\tfrac12))^{1-2\gamma_*} \left( 1 - h^{2} (n+\tfrac12)^{2} \right)_+^{\gamma_*}
			\leq \int_0^1 t^{1-2\gamma_*}(1-t^2)^{\gamma_*} \dd t \,.
		\end{equation}
		Equality holds if and only if $h=2\sqrt{1-\gamma_*}\,$.
	\end{theorem}
	
The remainder of this paper consists of two sections, devoted to the proofs of Theorems \ref{thm:endpoint} and \ref{thm:threshold}, respectively.
	

\section{Proof of Theorem~\ref{thm:endpoint}}\label{sec:thm3dproof}

Let us introduce, for $h>0\,$,  
\begin{equation}\label{eq:defE}
E(h) := h\sum_{n=1}^{\infty}(hn)^{2/3}(1-h^2n^2)_+^{2/3} - \int_0^1t^{2/3}(1-t^2)^{2/3}\dd t \,.
\end{equation}
Theorem 4 is equivalent to the assertion that \(E(h)<0\) for all \(h>0\). We shall prove the following quantitative version of this negativity.

\begin{proposition}\label{prop:quant3}
	For all $0<h\leq 1$, we have
	\begin{equation}\label{quant}
		E(h)< -\frac{\Gamma(5/3) \, \zeta(5/3)}{10 \, (2\pi)^{5/3}} \, h^{5/3} \,.
	\end{equation}
\end{proposition}

Numerically, we find $\frac{\Gamma(5/3) \, \zeta(5/3)}{10 \, (2\pi)^{5/3}} \approx 0.008960$. We do not expect this bound to be close to optimal. Our results from \cite{FHN}, however, show that the order $h^{5/3}$ is optimal as $h\to 0\,$.

Clearly, Proposition \ref{prop:quant3} implies Theorem \ref{thm:endpoint} in the range $0<h\leq 1$. Since all summands in the sum in the definition of $E(h)$ vanish for \(h\geq1\) while the integral is positive, we have $E(h)<0$ for $h\geq 1$. Thus, Theorem \ref{thm:endpoint} follows, once Proposition \ref{prop:quant3} has been proved.

The proof of Proposition \ref{prop:quant3} relies on an integral representation formula for $E(h)$ of Abel--Plana type. For background and connections with other summation formulas, we refer to \cite{BFSS}. Finite-interval versions of the Abel–Plana formula, which is what we need, have also been used in theoretical physics studies of the Casimir effect; see \cite[Section 2]{Sa} or \cite[Appendix B]{INS}. We give a self-contained derivation of the formulas needed here.

It will be convenient to parametrize $h=2\pi/a$ and work with the variable $a$. The representation formula involves the function
\begin{equation}\label{eq:defF}
F(z):=z^\frac23 (1-z^2)^\frac23,
\end{equation}
which is holomorphic on the strip $0 < \re z <1\,$.

\begin{lemma}\label{lem:representation3}
	For all $a>0$, we have, with $E$ defined in \eqref{eq:defE} and $F$ defined in \eqref{eq:defF}, 
	\[
	E(2\pi/ a) =-\sqrt3\int_0^\infty\frac{y^{2/3}(1+y^2)^{2/3}}{e^{ay}-1}\dd y
	+2\re\int_0^\infty\frac{-iF(1+iy)}{e^{ay-ia}-1}\dd y\,.	
	\]
\end{lemma}

\begin{proof}
	Given $a>0$ we write $h:=2\pi/a$. For $\varepsilon\in(0,1/2)$ and $R>0$, we let \(\Gamma_{\varepsilon,R}\) be the positively oriented boundary of \(\{\varepsilon<\re z<1-\varepsilon\,,\ |\im z|<R\}\). We only consider $\varepsilon$ such that neither vertical side passes through a point of \(h\mathbb Z\).
	The function
	\[
	G_h(z):=\pi F(z)\cot\!\left(\frac{\pi z}{h}\right)
	\]
	has residue \(hF(nh)\) at \(z=nh\)\,. Hence, by the residue theorem,
	\[
	\frac{1}{2\pi i}\int_{\Gamma_{\varepsilon,R}}G_h(z)\dd z
	=h\sum_{\varepsilon<nh<1-\varepsilon}F(nh) \,.
	\]
	
	We write \(\Gamma_{\varepsilon,R}^{+}\) for the upper half of the contour,
	oriented from \(1-\varepsilon\) to \(\varepsilon\). Using
	\[
	G_h(z)=-i\pi F(z)
	-2\pi i\,\frac{F(z)}{e^{-iaz}-1} \,,
	\]
	Cauchy's Theorem and complex conjugation, noting \(F(\bar z)=\overline{F(z)}\), give
	\[
	h\sum_{\varepsilon<nh<1-\varepsilon}F(nh)
	-\int_{\varepsilon}^{1-\varepsilon}F(x)\dd x
	=-2\re\int_{\Gamma_{\varepsilon,R}^{+}}
	\frac{F(z)}{e^{-iaz}-1}\dd z \,.
	\]
	
	The upper horizontal integral is \(\mathcal O(R^2e^{-aR})\).
	Moreover, \(F\) vanishes to order \(2/3\) at both endpoints, so the integrands on the vertical sides are \(O(y^{-1/3})\) near zero, uniformly for sufficiently small \(\varepsilon>0\).
	In particular, when \(a\in 2\pi\mathbb N\), we have \(e^{ay-ia}-1\sim ay\) as \(y\downarrow0\), so the corresponding singularity is integrable.
	Letting first \(R\to\infty\), then \(\varepsilon\to0\), yields
	\[
	E(h)=2\re\left(
	i\int_0^\infty\frac{F(iy)}{e^{ay}-1}\dd y
	-i\int_0^\infty\frac{F(1+iy)}{e^{ay-ia}-1}\dd y
	\right).
	\]
	Since by \eqref{eq:defF}, \(F(iy)=e^{i\pi/3}y^{2/3}(1+y^2)^{2/3}\), we obtain the claimed formula.
\end{proof}~

\begin{proof}[Proof of Proposition \ref{prop:quant3}]~\\

	\emph{Step 1. Upper bound from the representation formula.} Applying the identity
	\begin{equation}\label{eq:step1proof}
	\left|\frac1{re^{-i\theta}-1}-\frac1{r^2-1}\right|=\frac r{r^2-1}
	\qquad\text{for all}\ r>1
	\end{equation}
	with \(r=e^{ay}\) and \(\theta=a\), we obtain
	\[
	\begin{aligned}
		2\operatorname{Re}\frac{-iF(1+iy)}{e^{ay-ia}-1}
		={}&
		2\operatorname{Re}\frac{-iF(1+iy)}{e^{2ay}-1}\\
		&+2\operatorname{Re}\left[
		-iF(1+iy)\left(
		\frac1{e^{ay-ia}-1}-\frac1{e^{2ay}-1}
		\right)\right]\\
		\le{}&
		\frac{2\operatorname{Im}F(1+iy)}{e^{2ay}-1}
		+\frac{2e^{ay}|F(1+iy)|}{e^{2ay}-1}\\
		={}&
		\frac{2\operatorname{Im}F(1+iy)}{e^{2ay}-1}
		+\frac{|F(1+iy)|}{\sinh(ay)} \,.
	\end{aligned}	
	\]
	Inserting this into the identity from Lemma \ref{lem:representation3} and bounding \(y^{2/3}(1+y^2)^{2/3}\geq y^{2/3}\) gives
	\begin{equation}\label{circlebound}
		E(2\pi/a)\leq
		-\sqrt3\int_0^\infty\frac{y^{2/3}}{e^{ay}-1}\dd y
		+\int_0^\infty\frac{2\Ima F(1+iy)}{e^{2ay}-1}\dd y
		+\int_0^\infty\frac{|F(1+iy)|}{\sinh(ay)}\dd y.
	\end{equation}
	In the remainder of the proof, we will estimate the three integrals on the right side.

	\medskip
	
	\emph{Step 2. Bounds on \(F(1+iy)\) and integration.}
We shall use the bounds
\begin{equation}\label{Vbounds}
	\begin{aligned}
			|F(1+iy)|&\leq 2^{2/3}y^{2/3}\left(1+\frac5{12}y^2\right),\\
			2\Ima F(1+iy)&\leq
			2^{2/3}y^{2/3}\left(-\sqrt3+y+y^2\right).
	\end{aligned}
\end{equation}
Let us justify these inequalities. For the first one, we write \(F(1+iy)=2^{2/3}y^{2/3}e^{-i\pi/3}Q(y)\) with $Q(y) = (1+iy)^{2/3} (1+iy/2)^{2/3}$. The bound is then a consequence of \(|Q(y)|=[(1+y^2)(1+y^2/4)]^{1/3}\) and
\[
\left(1+\frac5{12}\, y^2\right)^3
-(1+y^2) \left(1+\frac14\, y^2\right)
=\frac{13}{48} \, y^4+\frac{125}{1728} \, y^6\ge0.
\]
For the second bound, we note that \(Q(0)=1\), \(Q'(0)=i\) and \(|Q''(y)|\le\frac79<1\) for all $y\geq 0$, so by Taylor's theorem
\[
|Q(y)-1-iy|\le\frac12y^2\qquad\text{for all}\ y\geq 0 \,.
\]
This easily implies the second bound in \eqref{Vbounds}.

For \(q>0\), define
\begin{equation}\label{eq:defmqnq}
M_q(a):=\int_0^\infty\frac{y^q}{e^{ay}-1}\dd y
\qquad\text{and}\qquad
N_q(a):=\int_0^\infty\frac{y^q}{\sinh(ay)}\dd y \,.
\end{equation}

Bounding \(y^{2/3}(1+y^2)^{2/3}\geq y^{2/3}\), equations \eqref{circlebound} and \eqref{Vbounds} give
\begin{align}
	E(2\pi/a)& \leq
	-\sqrt3 \, M_{2/3}(a)- 2^{2/3} \,\sqrt3 \, M_{2/3}(2a)
	+2^{2/3} \, M_{5/3}(2a)+2^{2/3} \, M_{8/3}(2a)\notag\\
	& \quad +2^{2/3} \, N_{2/3}(a)+\frac{2^{2/3}\, 5}{12} \, N_{8/3}(a) \notag\\
	&= -\left[ 1+\frac{3\sqrt3}{2}-2^{5/3} \right] M_{2/3}(a)+\frac14 \, M_{5/3}(a) \notag \\
	& \quad +\left[\frac18+\frac{2^{2/3}\, 5}{6}(1-2^{-11/3})\right]M_{8/3}(a) \,. \label{collected}
\end{align}
Here the equality comes from the following two identities:
\begin{equation}\label{eq:identitiesmqnq}
	M_q(2a)=2^{-q-1}M_q(a)
	\qquad\text{and}\qquad
	N_q(a) =2(1-2^{-q-1}) \, M_q(a) \,.
\end{equation}
Indeed, the first identity follows by a simple change of variables and the second one follows comparing the expansions
\[
\frac{1}{e^{ay}-1} = \sum_{j=0}^\infty e^{-(j+1)ay}
\qquad\text{and}\qquad
\frac1{\sinh(ay)}
=2\sum_{k=0}^{\infty}e^{-(2k+1)ay} \,.
\]

According to \cite[(25.5.1)]{NIST}, we have
\begin{equation}\label{eq:mqexplicit}
	M_q(a) =\frac{\Gamma(q+1) \, \zeta(q+1)}{a^{q+1}} \,.
\end{equation}
Using the fact that the zeta function is strictly decreasing on \((1,\infty)\), we get
\[
\frac{M_{5/3}(a)}{M_{2/3}(a)}
=\frac{5}{3a}\frac{\zeta(8/3)}{\zeta(5/3)}
<\frac5{3a} \,,
\qquad
\frac{M_{8/3}(a)}{M_{2/3}(a)}
=\frac{40}{9a^2}\frac{\zeta(11/3)}{\zeta(5/3)}
<\frac{40}{9a^2} \,.
\]
Since \(2^{2/3}<8/5\), we have
\[
\frac18+\frac{2^{2/3}\,5}{6}(1-2^{-11/3})
<\frac18+\frac43=\frac{35}{24}.
\]
Moreover, the bounds $\sqrt3>17/10$ and, again, $2^{2/3}<\frac85$ imply
\[
1+\frac{3\sqrt3}{2}-2^{5/3} >\frac7{20}.
\]

We conclude that
\begin{equation}\label{finalbound}
	E(2\pi/a)\leq M_{2/3}(a)
	\left[- \frac7{20} +\frac5{12a}+\frac{175}{27a^2}\right].
\end{equation}

\medskip

\emph{Step 3. Conclusion.}
For \(0<h\leq1\), we have \(a=2\pi/h\geq2\pi>6\). 
Consequently,
\[
- \frac7{20} +\frac5{12a}+\frac{175}{27a^2}
<-\frac7{20}+\frac5{72}+\frac{175}{972}
=-\frac{977}{9720}<-\frac1{10} \,.
\]
This bound, together with \eqref{finalbound} and the fact that, by \eqref{eq:mqexplicit}, \(M_{2/3}(a) = \frac{\Gamma(5/3)\,\zeta(5/3)}{(2\pi)^{5/3}}h^{5/3}\), proves Proposition \ref{prop:quant3}.
\end{proof}


\section{Proof of Theorem \ref{thm:threshold}}\label{sec:thm2dproof}

We recall that \(\g\) is the unique solution of \eqref{eq:root}, as discussed in Subsection \ref{sec:gammastar}. We will use the same notation $E(h)$ as in the previous section in \eqref{eq:defE} to denote the analogous quantity that we want to bound in Theorem \ref{thm:threshold}, that is,
\begin{equation}\label{eq:defE2}
E(h) := h\sum_{n=0}^{\infty}\bigl(h(n+\tfrac12)\bigr)^{1-2\g}
\bigl(1-h^2(n+\tfrac12)^2\bigr)_+^{\g}
-\int_0^1t^{1-2\g}(1-t^2)^{\g}\dd t \,.
\end{equation}
This should not cause any confusion, but rather emphasize the parallel nature of the proofs. Theorem~\ref{thm:threshold} is equivalent to the assertion that \(E(h)\leq 0\) for all \(h>0\), with equality if and only if \(h=2\sqrt{1-\gamma_*}\). We shall prove the following quantitative estimate, which involves the Dirichlet eta function, defined, for $p>0$, by
\begin{equation}\label{eq:defeta}
\eta(p):=\sum_{n\geq1}(-1)^{n-1}n^{-p} \,.
\end{equation}

\begin{proposition}\label{prop:quant2}
	For all $0<h\leq 2/3$, we have
	\begin{equation}\label{quant2}
		E(h)< -\frac{3 \, \Gamma(2-2\g) \, \eta(2-2\g)}{100 \, (2\pi)^{2-2\g}} \, h^{2-2\g} \,.
	\end{equation}
\end{proposition}

Numerically, we find $\frac{3 \, \Gamma(2-2\g) \, \eta(2-2\g)}{100 \, (2\pi)^{2-2\g}} \approx 0.005653$. We do not expect this bound to be close to optimal. Our results from \cite{FHN}, however, show that the order $h^{2-2\gamma_*}$ is optimal for $h\to 0$.

By combining the terms $n=2k-1$ and $n=2k$ in the sum defining $\eta(p)\,$, we see that $\eta(p)>0$ for all $p>0$. Therefore, Proposition \ref{prop:quant2} proves Theorem \ref{thm:threshold} in the range $0<h\leq 2/3$. In the range $2/3\leq h<2$ only the first term in the sum defining $E(h)$ can be non-zero. Setting $x=h^2/4$, the beta integral, \eqref{eq:root} and a simple maximization with respect to $x\in(0,1)$ give
\[
E(h)=2x^{1-\g}(1-x)^{\g}-\frac{\pi\g}{2\sin(\pi\g)}\leq 0 \,.
\]
Equality holds if and only if $x=1-\g$, that is, $h=2\sqrt{1-\g}$; this value belongs to $(2/3,2)$ by \eqref{eq:gammastarbounds} below. Finally, for $h\geq 2$ all terms in the sum defining $E(h)$ vanish, while the integral is positive. Thus, Theorem \ref{thm:threshold} follows, once Proposition \ref{prop:quant2} has been proved.

As in the proof of Proposition \ref{prop:quant3}, our argument relies on an integral representation formula for \(E(h)\). Again we will use the parametrization $h=2\pi/a$ and work with the variable $a$. Using the same notation for corresponding objects we now set (compare with \eqref{eq:defF})
\begin{equation}\label{eq:defF2}
F(z):=z^{1-2\g}(1-z^2)^{\g} \,,
\end{equation}
which is holomorphic on the strip $0 < \re z <1$.

\begin{lemma}\label{lem:representation2}
	For all $a>0$, we have
	\[
	E(2\pi/a)=2\cos(\pi\g)\int_0^\infty\frac{y^{1-2\g}(1+y^2)^{\g}}{e^{ay}+1}\dd y
	-2\im\int_0^\infty\frac{F(1+iy)}{e^{ay-ia}+1}\dd y \,.
	\]
\end{lemma}

\begin{proof}
	The proof is essentially the same as that of Lemma \ref{lem:representation3}, replacing the cotangent with the tangent.
	Given $a>0$, we write $h:=2\pi/a$. We use the contours $\Gamma_{\varepsilon,R}$ and $\Gamma_{\varepsilon,R}^+$ from the proof of Lemma \ref{lem:representation3}, with neither vertical side passing through a point of $h(\mathbb Z+\tfrac12)$. The function
	\[
	G_h(z):=-\pi F(z)\tan\!\left(\frac{\pi z}{h}\right)
	=-i\pi F(z)+2\pi i\,\frac{F(z)}{e^{-iaz}+1}
	\]
	has residue $hF(h(n+\tfrac12))$ at $z=h(n+\tfrac12)$. As in that proof, the residue theorem, Cauchy's theorem and complex conjugation, noting \(F(\bar z)=\overline{F(z)}\), give
	\[
	h\sum_{\varepsilon<h(n+\frac12)<1-\varepsilon}F(h(n+\tfrac12))
	-\int_{\varepsilon}^{1-\varepsilon}F(x)\dd x
	=2\re\int_{\Gamma_{\varepsilon,R}^{+}}\frac{F(z)}{e^{-iaz}+1}\dd z \,.
	\]
	The upper horizontal integral is $\mathcal O(Re^{-aR})$. Near zero, the vertical integrands are uniformly $\mathcal O(1+y^{1-2\g})$ on the left and $\mathcal O(y^{\g-1})$ on the right, both integrable. The right-hand bound includes the case $e^{-ia}=-1$, when the denominator has a simple zero at $z=1$. Letting first $R\to\infty$, then $\varepsilon\to0$, yields
	\[
	E(h)=2\re\left(
	-i\int_0^\infty\frac{F(iy)}{e^{ay}+1}\dd y
	+i\int_0^\infty\frac{F(1+iy)}{e^{ay-ia}+1}\dd y
	\right).
	\]
	Since $F(iy)=ie^{-i\pi\g}y^{1-2\g}(1+y^2)^{\g}$, we obtain the claimed formula.
\end{proof}

\begin{proof}[Proof of Proposition \ref{prop:quant2}]~\\

	\emph{Step 1. Upper bound from the representation formula.} Arguing similarly as in Step 1 of the proof of Proposition \ref{prop:quant3} and using \(\cos(\pi\g)<0\) (as a consequence of $1/2<\g<1$, which follows easily from the discussion in Subsection \ref{sec:gammastar} since the left side of \eqref{eq:root} is $1/2$, which exceeds $\pi/8$), we deduce from the identity in Lemma \ref{lem:representation2} the inequality
	\begin{equation}\label{circle}
		\begin{aligned}
			E(2\pi/a)\leq{}&2\cos(\pi\g)\int_0^\infty\frac{y^{1-2\g}}{e^{ay}+1}\dd y\\
			&+\int_0^\infty\frac{2\Ima F(1+iy)}{e^{2ay}-1}\dd y
			+\int_0^\infty\frac{|F(1+iy)|}{\sinh(ay)}\dd y \,.
		\end{aligned}
	\end{equation}
	In the remainder of the proof, we will estimate the three integrals on the right side.
	
	\medskip
	
	\emph{Step 2. Bounds on \(F(1+iy)\) and integration.}	
	We shall use the inequalities
	\begin{equation}\label{Vbounds2}
		\begin{aligned}
			|F(1+iy)|&\leq 2^{\g}y^{\g}\left(1+\frac{y^2}{12}\right),\\
			2\Ima F(1+iy)&\leq 2^{1+\g}y^{\g}
			\left(-\sin\frac{\pi\g}{2}+\lambda\left(\cos\frac{\pi\g}{2}\right)y+Dy^2\right) \,.
		\end{aligned}
	\end{equation}
	Here
	\[
	\lambda:=1-\frac32\g,
	\qquad D:=\frac{\g(23\,\g-11)}8 \,.
	\]
	Let us justify these bounds. We write
	\[
	F(1+iy)=2^{\g}y^{\g}e^{-i\pi\g/2}Q(y),
	\qquad Q(y):=(1+iy)^{1-2\g}(1+iy/2)^{\g} \,.
	\]
	By \eqref{eq:gammastarbounds} and concavity,
	\[
	|Q(y)|=(1+y^2)^{1/2-\g}(1+y^2/4)^{\g/2}
	\leq(1+y^2/4)^{\g/2}
	\leq1+\frac{\g}{8}y^2\leq1+\frac{y^2}{12} \,.
	\]
	For the second bound, note that $Q(0)=1$ and $ Q'(0)=i\lambda$. Differentiating twice and using $1\leq|1+iy/2|\leq|1+iy|$, we obtain
	\[
	|Q''(y)|\leq 2\g(2\g-1)+\g(2\g-1)+\frac{\g(1-\g)}4=2D \,.
	\]
	Thus Taylor's formula gives
	\[
	|Q(y)-1-i\lambda y|\leq Dy^2
	\qquad\text{for all}\ y\geq0 \,.
	\]
	Taking imaginary parts after multiplication by $e^{-i\pi\g/2}$ proves the second bound in \eqref{Vbounds2}.
	
	Let $M_q$ and $N_q$ be as in \eqref{eq:defmqnq}, and for $q>-1$ define
	\[
	L_q(a):=\int_0^\infty\frac{y^q}{e^{ay}+1}\dd y \,.
	\]
	Equations \eqref{circle} and \eqref{Vbounds2} give
	\begin{align}
		E(2\pi/a)\leq{}&2\cos(\pi\g) \, L_{1-2\g}(a)
		-2^{1+\g}\left(\sin\frac{\pi\g}{2}\right)M_{\g}(2a)\notag\\
		&+2^{1+\g}\lambda\left(\cos\frac{\pi\g}{2}\right)M_{\g+1}(2a)
		+2^{1+\g}D\,M_{\g+2}(2a)\notag\\
		&+2^{\g}\,N_{\g}(a)+\frac{2^{\g}}{12} \, N_{\g+2}(a)\notag\\
		={}&2\cos(\pi\g) \, L_{1-2\g}(a)
		+\left(2^{1+\g}-1-\sin\frac{\pi\g}{2}\right)M_{\g}(a)\notag\\
		&+\frac{\lambda}{2}\cos\frac{\pi\g}{2} \, M_{\g+1}(a)
		+\left[\frac D4+\frac{2^{\g}}6(1-2^{-\g-3})\right]M_{\g+2}(a) \,.
		\label{collected2}
	\end{align}
	Here the equality follows from the two identities in \eqref{eq:identitiesmqnq}.
	
	In what follows, we need numerical estimates on the coefficients appearing in \eqref{collected2}. They rely on the enclosure
	\begin{equation}\label{eq:gammastarbounds}
		0.61<\g<0.62 \,,
	\end{equation}
	which can be justified as follows. In view of Subsection \ref{sec:gammastar}, it suffices to check
	\[
	\begin{aligned}
		2\left(\frac{61}{100}\right)^{61/100}
		\left(\frac{39}{100}\right)^{39/100}
		&>\frac{51}{50}>
		\frac{61\pi}{200\sin(61\pi/100)},\\
		2\left(\frac{31}{50}\right)^{31/50}
		\left(\frac{19}{50}\right)^{19/50}
		&<\frac{103}{100}<
		\frac{31\pi}{100\sin(31\pi/50)} \,.
	\end{aligned}
	\]
	The comparisons involving powers follow by raising to the powers $100$ and $50$, respectively. The other two follow from $157/50<\pi<22/7$ and
	\[
	1-\frac{x^2}{2}\leq\cos x\leq1-\frac{x^2}{2}+\frac{x^4}{24},
	\qquad \sin(\pi\gamma)=\cos\bigl(\pi(\gamma-\tfrac12)\bigr).
	\]
	This proves \eqref{eq:gammastarbounds}.
	
	Using \eqref{eq:mqexplicit}, \eqref{eq:gammastarbounds}, and the strict decrease of the zeta function on $(1,\infty)$, we get
	\[
	\begin{aligned}
		\frac{M_{\g+1}(a)}{M_{\g}(a)}
		&=\frac{\g+1}{a}\frac{\zeta(\g+2)}{\zeta(\g+1)}
		<\frac{1.62}{a},\\
		\frac{M_{\g+2}(a)}{M_{\g}(a)}
		&=\frac{(\g+1)(\g+2)}{a^2}\frac{\zeta(\g+3)}{\zeta(\g+1)}
		<\frac{4.245}{a^2} \,.
	\end{aligned}
	\]
	Moreover, \eqref{eq:gammastarbounds} implies
	\[
	\begin{aligned}
		2\cos(\pi\g)&<2\cos(0.61\pi)<-0.67 \,,\\
		2^{1+\g}-1-\sin\frac{\pi\g}{2}&<2^{1.62}-1-\sin(0.305\pi)<1.26 \,,\\
		0<\frac{\lambda}{2}\cos\frac{\pi\g}{2}
		&<\frac{0.085}{2}\cos(0.305\pi)<0.025\,,\\
		& D < 0.26 \,,\\
		0<\frac D4+\frac{2^{\g}}6(1-2^{-\g-3})
		&=\frac D4+\frac{2^{\g}}6-\frac1{48}
		<\frac{0.26}{4}+\frac{1.537}{6}-\frac1{48}<0.31 \,.
	\end{aligned}
	\]
	The elementary numerical comparisons can also be checked using
	$223/71<\pi<22/7$, the alternating Taylor bounds for sine and cosine, and $2^{31}<1.537^{50}$.
	Thus, for every $a>0$, we obtain from \eqref{collected2} that
	\begin{equation}\label{finalbound2}
		E(2\pi/a)<-0.67 \, L_{1-2\g}(a)
		+\left(1.26+\frac{0.0405}{a}+\frac{1.31595}{a^2}\right)M_{\g}(a) \,.
	\end{equation}
	
	\medskip
	
	\emph{Step 3. Conclusion.} We now assume $0<h\leq2/3$. Then $a=2\pi/h\geq3\pi>9$, and
	\[
	1.26+\frac{0.0405}{a}+\frac{1.31595}{a^2}
	<1.26+\frac{0.0405}{9}+\frac{1.31595}{81}<1.29 \,.
	\]
	Consequently, \eqref{finalbound2} gives
	\begin{equation}\label{eq:simplifiedbound2}
		E(2\pi/a)<-0.67 \, L_{1-2\g}(a)+1.29 \, M_{\g}(a) \,.
	\end{equation}
	In Step 4, we shall prove that
	\begin{equation}\label{eq:finalestimate}
		1.29 \, M_{\g}(a) <0.64 \, L_{1-2\g}(a) \,.
	\end{equation}
	
	We note that
	\begin{equation}\label{eq:lq}
		L_{q}(a)=\Gamma(q+1) \, \eta(q+1) \, a^{-q-1} \,,
	\end{equation}
	which follows by expanding
	\[
	\frac1{e^{ay}+1}
	=\sum_{j=1}^{\infty}\left(e^{-(2j-1)ay}-e^{-2jay}\right)
	\]
	and integrating the nonnegative summands.
	
	Combining \eqref{eq:simplifiedbound2} and \eqref{eq:finalestimate} and using \eqref{eq:lq}, we conclude that
	\[
	E(h)<-0.03 \, L_{1-2\g}(a)
	=-\frac{3\, \Gamma(2-2\g)\, \eta(2-2\g)}{100 \, (2\pi)^{2-2\g}} \, h^{2-2\g}<0 \,.
	\]
	This is the claimed inequality \eqref{quant2}.
	
	\medskip
	
	\emph{Step 4. Technical estimates.}
	It remains to justify the numerical bound \eqref{eq:finalestimate}. We begin by establishing the bounds
	\begin{equation}\label{eq:numericalbounds2}
		\Gamma(2-2\g)>1.18\,,\qquad \eta(2-2\g)>0.58\,,\qquad
		\Gamma(\g+1)\zeta(\g+1)<2.1 \,.
	\end{equation}
	
	For the proof of these bounds, we use $0.76<2-2\g<0.78$, which follows from \eqref{eq:gammastarbounds}. For the second bound, the alternating-series estimate gives
	\[
	\begin{aligned}
		\eta(2-2\g)&\geq(2-2\g)\sum_{j=1}^8\int_{2j-1}^{2j}t^{2\g-3}\dd t\\
		&>\frac{0.76}{0.78}\sum_{j=1}^8
		\bigl((2j-1)^{-0.78}-(2j)^{-0.78}\bigr)>0.5848 \,.
	\end{aligned}
	\]
	Also, \eqref{eq:gammastarbounds}, monotonicity of the zeta function on $(1,\infty)$ and the integral estimate for the tail of the zeta series give
	\[
	\zeta(\g+1)<\zeta(1.61)
	<\sum_{n=1}^{20}n^{-1.61}+\frac{20^{-0.61}}{0.61}<2.264 \,.
	\]
	Finally, $\Gamma$ is decreasing on $(0,1)$ and log-convex on $(0,\infty)$, so
	\[
	\Gamma(2-2\g)>\Gamma(0.78)>1.18,
	\qquad
	\Gamma(\g+1)\leq\max\{\Gamma(1.61),\Gamma(1.62)\}<0.90 \,.
	\]
	For completeness, these numerical gamma bounds follow from Stirling's inequalities \cite[(5.6.1)]{NIST},
	\[
	0<\log\Gamma(x)-(x-\tfrac12)\log x+x-\tfrac12\log(2\pi)
	<\frac1{12x}\qquad(x>0),
	\]
	after shifting the argument by $20$ and using
	$\Gamma(x+20)=\Gamma(x)\prod_{j=0}^{19}(x+j)$.
	This gives, more explicitly,
	\[
	\Gamma(0.78)>1.1827,\qquad
	\Gamma(1.61)<0.8947,\qquad
	\Gamma(1.62)<0.8960 \,.
	\]
	Since $0.90\cdot2.264<2.1$, this proves all three bounds in \eqref{eq:numericalbounds2}.
	
	By \eqref{eq:mqexplicit} and $3\g-1>0.83$, we therefore have
	\[
	\begin{aligned}
		\frac{1.29 \, M_{\g}(a)}{L_{1-2\g}(a)}
		&=\frac{1.29 \, \Gamma(\g+1)\zeta(\g+1)}{\Gamma(2-2\g)\,\eta(2-2\g)}
		a^{-(3\g-1)}\\
		&<\frac{1.29\cdot2.1}{1.18\cdot0.58}\,9^{-0.83}<0.64 \,.
	\end{aligned}
	\]
	The last comparison is equivalently the exact rational inequality
	\[
	\left(\frac{1.29\cdot2.1}{0.64\cdot1.18\cdot0.58}\right)^{100}<9^{83} \,.
	\]
	This proves \eqref{eq:finalestimate} and completes the proof of Proposition \ref{prop:quant2}.
\end{proof}
	
	
	\subsection*{AI declaration}
	
	The formulation of the main theorems, the proof strategy, including the use of Abel--Plana-type formulas, and initial numerical experiments are exclusively due to the authors. The implementation of these ideas was developed with the help of generative LLMs. All mathematical arguments and proofs in the final manuscript were checked and written by the authors, who take full responsibility.
	

\end{document}